\documentclass[reqno,draft]{amsart}

\theoremstyle{plain}
\newtheorem{theorem}{Theorem}
\newtheorem{lemma}[theorem]{Lemma}

\newtheorem{proposition}[theorem]{Proposition}

\theoremstyle{remark}

\def\om{\omega}

\def\Dd{\Delta}

\def\om{\omega}
\def\pp{\partial}
\def\na{\nabla}

\begin{document}
\title[2D MHD with mixed partial dissipation and magnetic diffusion]{Global regularity for the 2D MHD equations with mixed partial dissipation and magnetic diffusion}

\author{Chongsheng Cao}
\address{Department of Mathematics, Florida International University, Miami, FL 33199, USA}
\curraddr{}
\email{caoc@fiu.edu}
\thanks{}

\author{Jiahong Wu}
\address{Department of Mathematics, Oklahoma State University, Stillwater, OK 74078, USA}
\curraddr{}
\email{jiahong@math.okstate.edu}
\thanks{}



\date{}

\dedicatory{}

\begin{abstract}
Whether or not classical solutions of the 2D incompressible MHD equations without full dissipation and magnetic diffusion can develop finite-time singularities is a difficult issue.  A major result of this paper establishes the global regularity of classical solutions for the MHD equations with mixed partial dissipation and magnetic diffusion. In addition, the global existence, conditional regularity and uniqueness of a weak solution is obtained for the 2D MHD equations with only magnetic diffusion.
\end{abstract}

\maketitle

\vskip 0.1in
\section{Introduction}

This paper concerns itself with the fundamental issue of whether classical solutions of the 2D incompressible MHD equations can develop finite-time singularities. The 2D MHD equations under consideration assume the form
\begin{eqnarray}
&& u_t + u\cdot\nabla u =-\nabla p +\nu_1\,
 u_{xx} + \nu_2\, u_{yy} + b\cdot\nabla b,   \label{ueq}  \\
&& b_t + u\cdot\nabla b =\eta_1\; b_{xx} + \eta_2 \; b_{yy} + b\cdot\nabla u,  \label{beq}   \\
&& \nabla\cdot u=0,   \label{divu} \\
&& \nabla \cdot b=0, \, \label{divb}
\end{eqnarray}
where $(x,y)\in {\mathbf R}^2$, $t\ge 0$, $u=(u_1(x,y,t), \;u_2(x,y,t))$ denotes the 2D velocity field, $p=p(x,y,t)$ denotes the pressure, $b=(b_1(x,y,t),b_2(x,y,t))$ denotes the magnetic field,  and $\nu_1$, $\nu_2$, $\eta_1$ and $\eta_2$ are nonnegative real parameters.

\vskip.1in
When $\nu_1>0$, $\nu_2>0$, $\eta_1>0$ and $\eta_2>0$, (\ref{ueq})-(\ref{divb}) has a unique global classical solution for every initial data $(u_0, b_0)\in H^m$ with $m\ge 2$ (see e.g. \cite{DL},\cite{SeTe}). However, if any one of these parameters is zero, the global regularity issue has not been settled. This paper establishes the global regularity of classical solutions of (\ref{ueq})-(\ref{divb}) with either $\nu_1=0$, $\nu_2=\nu>0$, $\eta_1=\eta>0$ and $\eta_2=0$ or $\nu_1=\nu>0$, $\nu_2=0$, $\eta_1=0$ and $\eta_2=\eta>0$. More precisely, we
have the following theorem.
\begin{theorem} \label{major1}
Consider the 2D MHD equations (\ref{ueq})-(\ref{divb}) with $\nu_1=0$, $\nu_2=\nu>0$, $\eta_1=\eta>0$ and $\eta_2=0$. Assume $u_0\in H^2({\mathbf R}^2)$ and $b_0\in H^2({\mathbf R}^2)$ with $\nabla \cdot u_0=0$ and $\nabla \cdot b_0=0$. Then (\ref{ueq})-(\ref{divb}) with the initial data $(u_0, b_0)$ has a unique global classical solution $(u,b)$.  In addition, $(u, b)$ satisfies
\begin{equation}
(u, b) \in L^\infty([0,\infty); H^2), \qquad \om_y \in L^2([0,\infty); H^1), \qquad j_x \in L^2([0,\infty); H^1),
\end{equation}
where $\omega=\na \times u$ and $j=\na \times b$ represent the vorticity and the current density, respectively.
\end{theorem}
\noindent A similar global regularity result can also be stated for (\ref{ueq})-(\ref{divb}) with $\nu_1=\nu>0$, $\nu_2=0$, $\eta_1=0$ and $\eta_2=\eta>0$.

\vskip.1in
Attention is also paid to the 2D MHD equations without dissipation but with magnetic diffusion, namely (\ref{ueq})-(\ref{divb}) with $\nu_1=\nu_2=0$ but with $\eta_1=\eta_2=\eta>0$. In this case, we obtain the following global
{\it a priori} bound for $\om=\na\times u$ and  $j=\na \times b$,
$$
\|\om (t) \|^2_2 + \|j(t)\|^2_2 + \eta\,\int_0^t  \|\nabla j(\tau)\|^2_2\,d\tau\le C(\eta) (\|\om(0)\|^2_2 + \|j(0)\|^2_2) \quad\mbox{for $t\ge 0$},
$$
where $C(\eta)$ is a constant depending on $\eta$ only. One consequence of this global bound is the existence of a global $H^1$-weak solution. It is not clear if such weak solutions are unique or can be improved to global classical solutions. However, if we know the velocity field $u$ of a solution obeys
\begin{equation}\label{condd}
\sup_{p\ge 2} \; \frac{1}{\sqrt{p}} \,\int_0^T \|\na u(t)\|_p \; dt < \infty,
\end{equation}
then this solution actually becomes a classical solution on $[0,T]$ and two weak solutions with one of their velocities satisfying this bound must coincide on $[0,T]$. We remark that (\ref{condd}) is weaker than the standard condition $\int_0^T \|\na u(t)\|_\infty \;dt <\infty$ and, as some preliminary evidence shows, is more likely to be proven true for (\ref{ueq})-(\ref{divb}) with $\eta_1=\eta_2=\eta>0$.

\vskip.1in
This work is partially motivated by the recent progress made by Chae \cite{Ch}, Hou and Li \cite{HL} and Danchin and Paicu \cite{DP} on the 2D Boussinesq equations,
\begin{eqnarray}
&&
u_t  + u \cdot\na u =-\nabla p + \nu\, \Delta u + \theta \,e_2,
\label{bou1}  \\
&&
\nabla \cdot u  =0,  \label{bou2} \\
&&
\theta_t + u\cdot\na \theta = \eta\, \Delta \theta, \label{bou3}
\end{eqnarray}
where the 2D vector $u$ represents the velocity field, the scalar $\theta$ the temperature, and $e_2=(0,1)$. Chae \cite{Ch} and Hou and Li \cite{HL} independently established the global regularity of (\ref{bou1})-(\ref{bou2}) with either dissipation or thermal diffusion. Danchin and Paicu \cite{DP}  constructed global solutions of (\ref{bou1})-(\ref{bou2}) with either $\eta=0$ and $\nu \Delta u$ replaced by $\nu\, u_{xx}$ or $\nu=0$ and $\eta \Delta \theta$ by $\eta\, \theta_{xx}$. We remark that the global regularity issue for the 2D MHD equations (\ref{ueq})-(\ref{divb}) is more sophisticated. The equations of $u$ and $b$ in (\ref{ueq})-(\ref{divb}) are both nonlinearly coupled vectors equations and the approaches in \cite{Ch},\cite{DP} and \cite{HL} do not appear to apply. In fact, it is not clear if (\ref{ueq})-(\ref{divb}) with $\eta_1=\eta_2=0$ or (\ref{ueq})-(\ref{divb}) with $\nu_2=\eta_2=0$ has global classical solutions.

\vskip.1in
The rest of this paper is divided into two sections. The second section is devoted to the global regularity of (\ref{ueq})-(\ref{divb}) with either $\nu_1=0$, $\nu_2=\nu>0$, $\eta_1=\eta>0$ and $\eta_2=0$ or $\nu_1=\nu>0$, $\nu_2=0$, $\eta_1=0$ and $\eta_2=\eta>0$. The third section handles (\ref{ueq})-(\ref{divb}) with $\nu_1=\nu_2=0$ and $\eta_1=\eta_2=\eta>0$. Throughout these sections the $L^p$-norm of a function $f$ is denoted by $\|f\|_p$, the $H^s$-norm by $\|f\|_{H^s}$ and the norm in the Sobolev space $W^{s,p}$ by $\|f\|_{W^{s,p}}$.

\vskip.3in
\section{Mixed partial dissipation and magnetic diffusion}

This section proves Theorem \ref{major1} as well as a parallel result for the case when $\nu_1=\nu>0$, $\nu_2=0$, $\eta_1=0$ and $\eta_2=\eta>0$. The proof of Theorem \ref{major1} is achieved through two stages. The first stage establishes a global bound for $\|\om(t)\|_2$ and $\|j(t)\|_2$ while the second obtains a bound for $\|\na \om(t)\|_2$ and $\|\na j(t)\|_2$. The following elementary lemma will play an important role.

\subsection{An elementary lemma}
\setcounter{theorem}{0}
\begin{lemma}  \label{LEMMA}
Assume that $f$, $g$, $g_y$, $h$ and $h_x$ are all in $L^2({\mathbf R}^2)$. Then,
\begin{equation} \label{EST}
 \iint | f \, g\, h|  \;dxdy
\leq C\; \|f\|_2 \; \|g\|_2^{1/2}   \; \|g_y\|_2^{1/2}   \; \|h\|_2^{1/2}
\; \| h_{x}\|_2^{1/2}.
\end{equation}
\end{lemma}

\begin{proof} Applying H\"{o}lder's inequality and the elementary inequality
\begin{equation} \label{ele}
\sup_{x\in{\mathbf R}} |F(x)| \le \sqrt{2}\; \left(\int |F(x)|^2 dx \right)^\frac14\; \left(\int |F_x(x)|^2 dx \right)^\frac14,
\end{equation}
we have
\begin{eqnarray}
&&\hskip-.3in
 \iint | f \, g\, h|  \;dxdy \nonumber \\
&&
\leq C \int \left[ \left( \int |f|^2 \; dx \right)^{1/2}
\left( \int |g|^2 \; dx \right)^{1/2} \left( \sup_{-\infty < x < \infty} h \right) \right]   \;dy  \nonumber\\
&&
\leq C \int \left[ \left( \int |f|^2 \; dx \right)^{1/2}
\left( \int |g|^2 \; dx \right)^{1/2}  \left( \int |h|^2 \; dx \right)^{1/4}
\left( \int |h_x|^2 \; dx \right)^{1/4}  \right]   \;dy \nonumber\\
&&
\leq C\; \|f\|_2  \; \left( \sup_{-\infty < y < \infty}
\left( \int |g|^2 \; dx \right)^{1/2} \right)\; \|h\|_2^\frac12\;
\|h_x\|_2^\frac12. \label{e1}
\end{eqnarray}
In addition, by (\ref{ele}) again,
\begin{eqnarray*}
&&\hskip-.3in
\sup_{-\infty < y < \infty}
\left( \int |g|^2 \; dx \right)^4 \\
&&
\leq C \left[ \int \left( \int |g|^2 \; dx   \right)^2 \; dy \right] \;
\left[ \int \left( \int |g| \, |g_y| \; dx \right)^2 \; dy \right] \\
&&
\leq C  \left( \int \left( \int |g|^4 \; dy   \right)^{1/2} \; dx\right)^2
\;  \int  \left[ \left(\int |g|^2 \; dx \right)
\left(\int |g_y|^2 \; dx \right) \right]\; dy  \\
&&
\leq C  \left( \int \left[ \left( \int |g|^2 \; dy   \right)^{3/4} \left( \int |g_y|^2 \; dy   \right)^{1/4} \right]
\; dx\right)^2 \\
&&  \hskip.3in   \times \left(\sup_{-\infty < y < \infty}\int |g|^2 \;dx \right)\;  \left( \iint |g_y|^2 \; dx dy \right) \\
&&
\leq C   \|g\|_2^{3} \; \|g_y\|_2
\;\left(\sup_{-\infty < y < \infty}\int |g|^2 \;dx \right)\; \|g_y\|_2^2.
\end{eqnarray*}
That is,
\begin{eqnarray}
&&
\sup_{-\infty < y < \infty}
\int |g|^2 \; dx
\leq C \;  \|g\|_2 \; \|g_y\|_2. \label{e2}
\end{eqnarray}
Combining (\ref{e1}) and (\ref{e2}) yields (\ref{EST}). This completes the proof of Lemma \ref{LEMMA}.
\end{proof}

\subsection{{\it A priori} bounds for $\|\om\|_2$ and $\|j\|_2$}

This subsection establishes {\it a priori} bounds for $\|\om\|_2$ and $\|j\|_2$ as stated in the following proposition.
\begin{proposition} \label{prop21}
If $(u,b)$ solves (\ref{ueq})-(\ref{divb}) with $\nu_1=0$, $\nu_2=\nu>0$, $\eta_1=\eta>0$ and $\eta_2=0$, then the vorticity $\om=\na\times u$ and the current density $j=\na \times b$ satisfy
\begin{equation}\label{omj}
\|\om(t)\|_2^2 + \|j(t)\|_2^2 + \nu\; \int_0^t \|\om_y (\tau)\|_2^2 \;d\tau + \eta\; \int_0^t \|j_x(\tau)\|_2^2 \; d\tau \le C (\nu,\eta)\; \left(\|\om_0\|_2^2 + \|j_0\|_2^2 \right)
\end{equation}
where $C(\nu,\eta)$ denotes a constant depending on $\nu$ and $\eta$ only, $\om_0=\nabla\times u_0$ and $j_0=\na \times b_0$.
\end{proposition}
\begin{proof}  Taking the inner products of (\ref{ueq}) with $u$ and (\ref{beq}) with $b$, adding the results and integrating by parts, we obtain
\begin{equation} \label{ub}
\|u(t)\|_2^2 + \|b(t)\|_2^2 + 2 \nu \int_0^t \|u_y(\tau)\|_2^2\,d\tau + 2 \eta \int_0^t \|b_x(\tau)\|_2^2 \,d\tau\, \le \|u_0\|_2^2 + \|b_0\|_2^2.
\end{equation}
Since $\omega$ and $j$ satisfy
\begin{eqnarray}
&& \om_t + u\cdot\nabla \om =\nu\; \om_{yy} +  b\cdot\nabla j  \label{vor1},\\
&& j_t + u\cdot\nabla j = \eta\, j_{xx} + b\cdot\nabla \om + 2\;\pp_xb_1(\pp_x u_2 + \pp_y u_1)-2\; \pp_x u_1 (\pp_x b_2 + \pp_y b_1),  \label{j1}
\end{eqnarray}
we find that $X(t)=\|\om(t)\|_2^2 + \|j(t)\|_2^2$ obeys
\begin{eqnarray*}
&&\hskip-.8in
\frac12\,\frac{d\,X(t)}{dt}  + \nu \, \|\om_y\|^2_2 +\eta\, \|j_x\|_2^2 \le 2\, \left|\int [\pp_xb_1(\pp_x u_2 + \pp_y u_1) \;-\; \pp_x u_1 (\pp_x b_2 + \pp_y b_1)]\;j \; dxdy \right|.
\end{eqnarray*}
Applying Lemma \ref{LEMMA}, we can bound the terms on the right as follows.  $C$'s in these estimates denote either pure constants or constants depending on $\nu$ and $\eta$ only.
\begin{eqnarray*}
&& \hskip-.3in
\int |\pp_x b_1|\; |\pp_x u_2|\; |j|  \;dxdy \\
&&
\leq C \|\pp_x u_2\|_2^{1/2}\; \|\pp_{xy}u_2\|_2^{1/2}   \; \|j\|_2^{1/2} \; \|j_x\|_2^{1/2}
\; \|\pp_x b_1\|_2 \\
&&
\leq \frac{\nu}{4}\, \|\pp_{xy} u_2\|_2^2 + \frac{\eta}{8} \|j_x\|_2^2 + C \|\pp_x u_2\|_2\,\|\pp_x b_1\|^2\|j\|_2 \\
&&
\leq \frac{\nu}{4}\, \|\omega_y\|_2^2 + \frac{\eta}{8} \|j_x\|_2^2 + C \|\omega\|_2\;\|\pp_x b_1\|_2^2\;\|j\|_2 \\
&&
\leq \frac{\nu}{4}\, \|\omega_y\|_2^2 + \frac{\eta}{8} \|j_x\|_2^2 + C\;
\|\pp_x b_1\|_2^2\,X(t).
\end{eqnarray*}

\begin{eqnarray*}
&&\hskip-.3in
\int  |\pp_x b_1|\, |\pp_y u_1|\, |j|\,dxdy \\
&&
\leq C\;\|\pp_x b_1\|_2^\frac12\; \|\pp_{xx} b_1\|_2^\frac12\; \|\pp_y u_1\|_2^\frac12\;\|\pp_{yy} u_1\|_2^\frac12\; \|j\|\\
&&
\leq  \frac{\nu}{4}\,\|\pp_{yy} u_1\|_2^2 + \frac{\eta}{8} \|\pp_{xx} b_1\|_2^2 +C \,\|\pp_x b_1\|\, \|\pp_y u_1\|_2\, \|j\|_2^2 \\
&&
\leq \frac{\nu}{4}\, \|\omega_y\|_2^2 + \frac{\eta}{8} \|j_x\|_2^2 + C\,( \|\pp_x b_1\|_2^2 + \|\pp_y u_1\|_2^2)\, \|j\|_2^2.
\end{eqnarray*}

\begin{eqnarray*}
&&\hskip-.3in
\left|\int  \pp_{x} u_1 \pp_{x}b_2 \,j\,dxdy \right| \\
&&
=\left|\int  (u_1 \pp_{xx} b_2 \,j + u_1\, \pp_x b_2 \,j_x)\;dxdy\right|\\
&&
\leq C\;\|u_1\|_2^\frac12\, \|\pp_y u_1\|_2^\frac12\, \|j\|_2^\frac12\, \|j_x\|_2^\frac12\, \|\pp_{xx} b_2\|_2 + C\;\|u_1\|_2^\frac12\, \|\pp_y u_1\|_2^\frac12\,\|\pp_x b_2\|_2^\frac12\, \|\pp_{xx} b_2\|_2^\frac12\, \|j_x\|_2\\
&&
\le C\;\|u_1\|_2^\frac12\, \|\pp_y u_1\|_2^\frac12\,\|j\|_2^\frac12\, \|j_x\|_2^\frac32 + C\; \|u_1\|_2^\frac12\, \|\pp_y u_1\|_2^\frac12\,\|\pp_x b_2\|_2^\frac12\,\|j_x\|_2^\frac32\\
&&
\le \frac{\eta}{8} \|j_x\|_2^2 + C \|u_1\|_2^2 \;\|\pp_y u_1\|_2^2 \;\|j\|_2^2 + C\,\|u_1\|_2^2\; \|\pp_y u_1\|_2^2\; \|\pp_x b_2\|_2^2 \\
&&
\le \frac{\eta}{8} \|j_x\|_2^2 + C \|u_1\|_2^2 \;\|\pp_y u_1\|_2^2 \;\|j\|_2^2.
\end{eqnarray*}

\begin{eqnarray*}
&&\hskip-.3in
\left|\int \pp_x u_1 \, \pp_y b_1\, j\,dx dy \right| \\
&&
\le \left|\int (u_1 \, \pp_{xy} b_1\, j + u_1 \, \pp_y b_1\, j_x)\; dxdy \right| \\
&&
\le C\;\|u_1\|_2^\frac12\; \|\pp_y u_1\|_2^\frac12\;\|j\|_2^\frac12\; \|j_x\|_2^\frac12 \; \|\pp_{xy} b_1\|_2 + C\;\|u_1\|_2^\frac12\; \|\pp_y u_1\|_2^\frac12\; \|\pp_y b_1\|_2^\frac12\; \|\pp_{xy} b_1\|_2^\frac12\; \|j_x\|_2 \\
&&
\le C\;\|u_1\|_2^\frac12\, \|\pp_y u_1\|_2^\frac12\,\|j\|_2^\frac12\, \|j_x\|_2^\frac32 + C\;\|u_1\|_2^\frac12\, \|\pp_y u_1\|_2^\frac12\, \|\pp_y b_1\|_2^\frac12\, \|j_x\|_2^\frac32   \\
&&
\le \frac{\eta}{8}\; \|j_x\|_2^2 + C\; \|u_1\|_2^2\; \|\pp_y u_1\|_2^2\; \|j\|_2^2 + C\,\|u_1\|_2^2\; \|\pp_y u_1\|_2^2\; \|\pp_y b_1\|_2^2 \\
&&
\le \frac{\eta}{8} \|j_x\|_2^2 + C \|u_1\|_2^2 \;\|\pp_y u_1\|_2^2 \;\|j\|_2^2.
\end{eqnarray*}
Combining these estimates, we have
\begin{eqnarray*}
&&\hskip-.8in
\frac{d\,X(t)}{dt}  + \nu \, \|\om_y\|^2_2 +\eta\, \|j_x\|_2^2 \le C\, (\|\pp_y u_1\|_2^2 + \|\pp_x b_1\|_2^2)\; X(t),
\end{eqnarray*}
which, together with (\ref{ub}), yields (\ref{omj}).
\end{proof}

\vskip.1in
\subsection{{\it A priori} bounds for $\|\na \om\|_2$ and $\|\na j\|_2$}

This subsection provides global {\it  a priori} bounds for $\|\na\om\|_2$ and $\|\na j\|_2$.
\begin{proposition}\label{prop22}
If $(u,b)$ solves (\ref{ueq})-(\ref{divb}) with $\nu_1=0$, $\nu_2=\nu>0$, $\eta_1=\eta>0$ and $\eta_2=0$, then the vorticity $\om$ and the current density $j$ satisfy
\begin{equation}\label{omj2}
\|\na \om(t)\|_2^2 + \|\na j(t)\|_2^2 + \nu\; \int_0^t \|\na\om_y (\tau)\|_2^2 \;d\tau + \eta\; \int_0^t \|\na j_x(\tau)\|_2^2 \; d\tau \le C (\nu,\eta)\; \left(\|\na \om_0\|_2^2 + \|\na j_0\|_2^2 \right)
\end{equation}
where $C(\nu,\eta)$ denotes a constant depending on $\nu$ and $\eta$ only.
\end{proposition}

\begin{proof}
Taking the inner products of (\ref{vor1}) with $\Delta \om$ leads to
\begin{eqnarray*}
&&
\frac12\,\frac{d\,}{dt} \|\na \om\|_2^2 +\; \nu \|\na \om_y\|_2^2 \\
&& \qquad
= -\int \na \om \cdot \na u \cdot \na \om \;dx dy + \int \na \om \cdot \na b \cdot \na j \;dx dy + \int b \cdot \na (\na j)\,\cdot \na \om \;dx dy.
\end{eqnarray*}
Similarly, taking the inner product of (\ref{j1}) with $\Dd j$ yields
\begin{eqnarray*}
&&
\frac12\,\frac{d\,}{dt} \|\na j\|_2^2 +\; \eta \|\na j_x\|_2^2 \\
&& \qquad
= -\int \na j \cdot \na u \cdot \na j \;dxdy + \int \na j \cdot \na b \cdot \na \om \;dxdy + \int b\cdot \na (\na \om)\cdot \na j \;dxdy \\
&& \qquad \quad
+ 2 \int \nabla \left[\pp_x b_1 (\pp_x u_2 + \pp_y u_1) \right] \cdot \nabla j \;dxdy
-2 \int \nabla \left[\pp_x u_1 (\pp_x b_2 + \pp_y b_1) \right] \cdot \nabla j \;dxdy.
\end{eqnarray*}
Adding the above equations and integrating by parts, we find
\begin{eqnarray*}
&&
\frac12\,\frac{d\,}{dt} (\|\na \om\|_2^2 + \|\na j\|_2^2) \;+\; \nu \|\na \om_y\|_2^2  + \eta \|\na j_x\|_2^2
=I_1 + I_2 +I_3 + I_4 + I_5,
\end{eqnarray*}
where
\begin{eqnarray*}
&&
I_1 = -\int \na \om \cdot \na u \cdot \na \om \;dx dy, \\
&&
I_2=-\int \na j \cdot \na u \cdot \na j \; dx dy, \\
&&
I_3= 2\int \na \om \cdot \na b \cdot \na j \;dx dy, \\
&&
I_4 =2 \int \nabla \left[\pp_x b_1 (\pp_x u_2 + \pp_y u_1) \right] \cdot \nabla j \; dx dy, \\
&&
I_5 = -2 \int \nabla \left[\pp_x u_1 (\pp_x b_2 + \pp_y b_1) \right] \cdot \nabla j \; dxdy.
\end{eqnarray*}
To bound $I_1$, we write the integrand explicitly and further divide it into four terms
\begin{eqnarray*}
I_1 &=& \int (\pp_x u_1  \;\om_x^2 + \pp_x u_2\;\om_x\;\om_y  + \pp_y u_1 \; \om_x \;\om_y + \pp_y u_2\; \om_y^2)\; dxdy \\
&=& I_{11} + I_{12} + I_{13} + I_{14}.
\end{eqnarray*}
By the divergence-free condition $\pp_x u_1 + \pp_y u_2 =0$ and Lemma \ref{LEMMA},
\begin{eqnarray*}
I_{11} &=& -\int \pp_y u_2 \;\om_x^2 \;dxdy \\
 &\le & C\; \|\pp_y u_2\|_2^\frac12\; \|\pp_{xy} u_2\|_2^\frac12\; \|\om_x\|_2^\frac12\; \|\om_{xy}\|_2^\frac12\; \|\om_x\| \\
&\le& C \; \|\om\|_2^\frac12\;  \|\om_y\|_2^\frac12\;\|\nabla \om_y\|_2^\frac12\; \|\nabla \om\|^\frac32_2 \\
&\le& \frac{\nu}{10} \;\|\nabla \om_y\|_2^2 +C\; \|\om\|_2^\frac23\; \|\om_y\|_2^\frac23\; \|\na \om\|_2^2.
\end{eqnarray*}
By Lemma \ref{LEMMA},
\begin{eqnarray*}
 I_{12} &\le& C \; \|\pp_x u_2\|_2^\frac12\; \|\pp_{xy} u_2\|_2^\frac12\; \|\pp_{y} \om\|_2^\frac12\; \|\pp_{xy} \om\|_2^\frac12\; \|\om_x\|_2\\
&\le& C\;  \|\om\|_2^\frac12\; \|\om_y\|_2^\frac12\; \|\nabla \om_y\|_2^\frac12\; \|\na \om\|^\frac32_2  \\
&\le&\frac{\nu}{10} \|\nabla \om_y\|_2^2 +C\; \|\om\|_2^\frac23\; \|\om_y\|_2^\frac23\; \|\na \om\|_2^2.
\end{eqnarray*}
$I_{13}$ and $I_{14}$ can be similarly bounded,
\begin{eqnarray*}
I_{13},\;I_{14} \le \frac{\nu}{10} \;\|\nabla \om_y\|_2^2 +C\; \|\om\|_2^\frac23\; \|\om_y\|_2^\frac23\; \|\na \om\|_2^2.
\end{eqnarray*}
$I_2$ and $I_3$ can be bounded by applying Lemma \ref{LEMMA}.
\begin{eqnarray*}
I_2 &\le& C\; \|\nabla u\|_2^\frac12\; \|\nabla u_y\|_2^\frac12\; \|\nabla j\|_2^\frac12\; \|\nabla j_x\|_2^\frac12\; \|\nabla j\|_2 \\
&\le&  C\; \|\om\|_2^\frac12\; \|\om_y\|_2^\frac12\; \|\nabla j\|_2^\frac32\; \|\nabla j_x\|_2^\frac12 \\
&\le&   \frac{\eta}{16}\; \|\nabla j_x\|_2^2 + C\;  \|\om\|_2^\frac23\; \|\om_y\|_2^\frac23\; \|\na j\|_2^2.
\end{eqnarray*}

\begin{eqnarray*}
I_3 &\le&  C\; \|\na b\|_2\; \|\na\om\|_2^\frac12\; \|\na \om_y\|_2^\frac12\;\|\nabla j\|_2^\frac12\; \|\nabla j_x\|_2^\frac12  \\
&\le& C\; \|j\|_2\; \|\na\om\|_2^\frac12\; \|\na \om_y\|_2^\frac12\;\|\nabla j\|_2^\frac12\; \|\nabla j_x\|_2^\frac12 \\
&\le& \frac{\nu}{10} \;\|\nabla \om_y\|_2^2 + \frac{\eta}{16}\; \|\nabla j_x\|_2^2  + C\;\|j\|^2_2\;\|\na\om\|_2\;\|\nabla j\|_2 \\
&\le& \frac{\nu}{10}\; \|\nabla \om_y\|_2^2 + \frac{\eta}{16}\; \|\nabla j_x\|_2^2  + C\;\|j\|^2_2\; (\|\na\om\|_2^2 +\|\nabla j\|_2^2).
\end{eqnarray*}
To bound $I_4$, we split it into two parts:
\begin{eqnarray*}
I_4 &=&  2\int\pp_x[\pp_x b_1 (\pp_x u_2 + \pp_y u_1)] \; j_x \;dxdy
+ 2 \int \pp_y[\pp_x b_1 (\pp_x u_2 + \pp_y u_1) ] \; j_y \;dxdy\\
& \equiv & I_{41} + I_{42}.
\end{eqnarray*}
Integrating by parts in $I_{41}$ and applying Lemma \ref{LEMMA}, we have
\begin{eqnarray*}
 I_{41} &=& -2 \int \pp_x b_1 (\pp_x u_2 + \pp_y u_1) \;  j_{xx}\;dxdy\\
&\le&  C \; \|\pp_x b_1\|_2^\frac12\;  \|\pp_{xx} b_1\|_2^\frac12\;
\|\pp_x u_2\|_2^\frac12\; \|\pp_{xy} u_2\|_2^\frac12\; \|j_{xx}\|_2 \\
&& +\; C\; \|\pp_x b_1\|_2^\frac12\;  \|\pp_{xx} b_1\|_2^\frac12\;
\|\pp_y u_1\|_2^\frac12\; \|\pp_{yy} u_1\|_2^\frac12\; \|j_{xx}\|_2 \\
&\le&  C\;\|j\|_2^\frac12\;  \|\nabla j\|_2^\frac12\;  \|\om\|_2^\frac12\; \|\om_y\|_2^\frac12\; \|\nabla j_x\|_2 \\
&\le& \frac{\eta}{16}\; \|\nabla j_x\|_2^2  + C\; \|\om\|_2 \; \|j\|_2\; \|\na\om\|_2\;  \|\nabla j\|_2\\
&\le&  \frac{\eta}{16} \;\|\nabla j_x\|_2^2  + C\; \|\om\|_2 \; \|j\|_2\; (\|\na\om\|_2^2 +  \|\nabla j\|_2^2).
\end{eqnarray*}
$I_{42}$ can be further decomposed into two parts:
\begin{eqnarray*}
I_{42} &=&  2\,\int \pp_{xy} b_1 (\pp_x u_2 + \pp_y u_1) ] \; j_y \;dxdy + 2 \int \pp_{x} b_1 (\pp_{xy} u_2 + \pp_{yy} u_1) ] \; j_y \;dxdy \\
&\equiv & I_{421} + I_{422}
\end{eqnarray*}
and these two terms can be bounded as follows.
\begin{eqnarray*}
I_{421} &\le & C \; \|\pp_{xy} b_1\|_2\; \|\pp_x u_2\|_2^\frac12\;  \|\pp_{xy} u_2\|_2^\frac12\; \|j_y\|_2^\frac12\; \|j_{xy}\|_2^\frac12\\
&& + \; C\; \|\pp_{xy} b_1\|_2\; \|\pp_y u_1\|_2^\frac12\;  \|\pp_{yy} u_1\|_2^\frac12\; \|j_y\|_2^\frac12\; \|j_{xy}\|_2^\frac12\\
&\le&  C\; \|\om\|_2^\frac12\; \|\om_y\|_2^\frac12\; \|\nabla j\|_2^\frac32\; \|\nabla j_x\|_2^\frac12 \\
&\le&  \frac{\eta}{16} \;\|\nabla j_x\|_2^2  + C\; \|\om\|_2^\frac23 \; \|\om_y\|_2^\frac23\; \|\nabla j\|_2^2
\end{eqnarray*}
\begin{eqnarray*}
I_{422} &\le&  C\; \|\pp_x b_1\|_2^\frac12\;\|\pp_{xy} b_1\|_2^\frac12\; \|\pp_{xy} u_2\|_2\;\|j_y\|_2^\frac12\; \|j_{xy}\|_2^\frac12\\
&& +\; C\; \|\pp_x b_1\|_2^\frac12\;\|\pp_{xy} b_1\|_2^\frac12\; \|\pp_{yy} u_1\|_2\;\|j_y\|_2^\frac12\; \|j_{xy}\|_2^\frac12\\
&\le&  C\; \|j\|_2^\frac12\;\|j_x\|_2^\frac12\; \|\om_y\|_2\;\|j_y\|_2^\frac12\; \|j_{xy}\|_2^\frac12 \\
&\le&  \frac{\eta}{16}\; \|\nabla j_x\|_2^2  + C\; \|j\|_2^\frac23\;\|j_x\|_2^\frac23\; \|\na \om\|_2^\frac43 \|\na j\|_2^\frac23\\
&\le & \frac{\eta}{16} \;\|\nabla j_x\|_2^2  + C\; \|j\|_2^\frac23\;\|j_x\|_2^\frac23\; (\|\na \om\|_2^2 + \|\na j\|_2^2).
\end{eqnarray*}
To bound $I_5$, we first write it into three terms,
\begin{eqnarray*}
I_5 &=&  -2 \int \pp_x \left[\pp_x u_1 (\pp_x b_2 + \pp_y b_1) \right]\;  j_x \;dxdy
-2 \int \pp_y \left[\pp_x u_1 (\pp_x b_2 + \pp_y b_1) \right]\;  j_y \;dxdy \\
&=& 2  \int \pp_x u_1 (\pp_x b_2 + \pp_y b_1) \;  j_{xx} \;dxdy  -2 \int
\pp_{xy} u_1\; (\pp_x b_2 + \pp_y b_1) \;  j_y \;dxdy \\
&&   -2 \int \pp_x u_1 (\pp_{xy} b_2 + \pp_{yy} b_1) \;  j_y \;dxdy \\
&\equiv & I_{51} + I_{52} + I_{53}.
\end{eqnarray*}
We bound these terms as follows.
\begin{eqnarray*}
I_{51} &\le&  C\; \|\pp_x u_1\|_2^\frac12\;\|\pp_{xy} u_1\|_2^\frac12\; \|\pp_{x} b_2\|_2^\frac12\;\|\pp_{xx} b_2\|_2^\frac12\; \|j_{xx}\|_2\\
&& + \;C\; \|\pp_x u_1\|_2^\frac12\;\|\pp_{xy} u_1\|_2^\frac12\; \|\pp_{y} b_1\|_2^\frac12\;\|\pp_{xy} b_1\|_2^\frac12\; \|j_{xx}\|_2\\
&\le & C\; \|\om\|_2^\frac12\; \|\na \om\|_2^\frac12\; \|j\|_2^\frac12\;\|\na j\|_2^\frac12\; \|\na j_x\|_2\\
&\le&  \frac{\eta}{16}\; \|\nabla j_x\|_2^2 + C\; \|\om\|_2\; \|j\|_2\;(\|\na \om\|_2^2 + \|\na j\|_2^2).
\end{eqnarray*}
\begin{eqnarray*}
I_{52} &\le & C\;  \|\pp_{xy} u_1\|_2\; \|\pp_x b_2\|_2^\frac12\;  \|\pp_{xy} b_2\|_2^\frac12\; \|j_y\|_2^\frac12\; \|j_{xy}\|_2^\frac12\\
&& + \;C\; \|\pp_{xy} u_1\|_2\; \|\pp_y b_1\|_2^\frac12\;  \|\pp_{yy} b_1\|_2^\frac12\; \|j_y\|_2^\frac12\; \|j_{xy}\|_2^\frac12\\
&\le&  C\;\|\om_y\|_2^\frac12\; \|\na \om\|_2^\frac12\;\|j\|_2^\frac12\; \|\na j\|_2\;\|\na j_x\|_2^\frac12\\
&\le&  \frac{\eta}{16}\; \|\nabla j_x\|_2^2 + C\;\|\om_y\|_2^\frac23\; \|j\|_2^\frac23\; \|\na \om\|_2^\frac23\; \|\na j\|_2^\frac43\\
&\le&  \frac{\eta}{16}\; \|\nabla j_x\|_2^2 + C\;\|\om_y\|_2^\frac23\; \|j\|_2^\frac23\; (\|\na \om\|_2^2 +  \|\na j\|_2^2).
\end{eqnarray*}
\begin{eqnarray*}
I_{53} &\le&  C\;  \|\pp_x u_1\|_2^\frac12\;\|\pp_{xy} u_1\|_2^\frac12\; \|\pp_{xy} b_2\|_2\;\|j_y\|_2^\frac12\; \|j_{xy}\|_2^\frac12\\
&& + C\; \|\pp_x u_1\|_2^\frac12\;\|\pp_{xy} u_1\|_2^\frac12\; \|\pp_{yy} b_1\|_2\;\|j_y\|_2^\frac12\; \|j_{xy}\|_2^\frac12\\
&\le & C\; \|\om\|_2^\frac12\;\|\om_y\|_2^\frac12\; \|\nabla j\|_2\;\|j_y\|_2^\frac12\; \|j_{xy}\|_2^\frac12 \\
&\le&  \frac{\eta}{16}\; \|\nabla j_x\|_2^2  + C\; \|\om\|_2^\frac23\;\|\om_y\|_2^\frac23\; \|\na j\|_2^2. \\
\end{eqnarray*}
Collecting the above estimates, we finally obtain
\begin{eqnarray*}
&&
\frac{d\,}{dt} (\|\na \om\|_2^2 + \|\na j\|_2^2) \;+\; \nu\; \|\na \om_y\|_2^2  + \eta\; \|\na j_x\|_2^2  \\
&&\qquad
\le C\; ((\|\om_y\|_2^\frac23+\|j_x\|_2^\frac23)\;(\|\om\|_2^\frac23 +\|j\|_2^\frac23)\; + \|j\|_2\;(\|\om\|_2+\|j\|_2))\; (\|\na \om\|_2^2 + \|\na j\|_2^2).
\end{eqnarray*}
Applying the bound from Proposition \ref{prop21}, we find
\begin{eqnarray*}
\|\na \om(t)\|_2^2 + \|\na j(t)\|_2^2 + \nu \int_0^t \|\na \om_y(\tau)\|_2^2\;d\tau
+ \eta  \int_0^t \|\na j_x(\tau)\|_2^2 \;d\tau
\le C(\nu, \eta)\; (\|\na \om_0\|_2^2 + \|\na j_0\|_2^2).
\end{eqnarray*}
This completes the proof of Proposition \ref{prop22}.
\end{proof}

\vskip.1in
\subsection{Proof of Theorem \ref{major1}}

This subsection presents the proof of Theorem \ref{major1}.

\begin{proof}[Proof of Theorem \ref{major1}]
With the {\it a priori} bounds of Propositions \ref{prop21} and \ref{prop22} at our disposal, the proof of this theorem can be achieved through a parabolic  regularization process. Let $\epsilon>0$ be a small parameter and consider a family of solutions $(u_\epsilon, b_\epsilon)$ satisfying the regularized system of equations
\begin{eqnarray}
&& \pp_t u_\epsilon + u_\epsilon \cdot\nabla u_\epsilon =-\nabla p_\epsilon + \nu\, \pp_{yy} u_\epsilon + b_\epsilon\cdot\nabla b_\epsilon + \epsilon\; \Dd u_\epsilon,   \label{ueq_ep}  \\
&& \pp_t b_\epsilon  + u_\epsilon\cdot\nabla b_\epsilon =\eta\; \pp_{xx} b_\epsilon + b_\epsilon\cdot\nabla u_\epsilon + \epsilon\; \Dd b_\epsilon  \label{beq_ep}   \\
&& \nabla\cdot u_\epsilon =0,   \label{divu_ep} \\
&& \nabla \cdot b_\epsilon =0, \, \label{divb_ep} \\
&& u_\epsilon(x,0) = \psi_\epsilon\ast u_0, \qquad  b_\epsilon(x,0) = \psi_\epsilon\ast b_0, \label{in_ep}
\end{eqnarray}
where $\psi_\epsilon(x) = \epsilon^{-2} \psi(x/\epsilon)$ with $\psi$ satisfying
$$
\psi\ge 0,\quad \psi \in C_0^\infty ({\mathbf R}^2)\quad\mbox{and}\quad \|\psi\|_1=1.
$$
Since $u_\epsilon(x,0)$ and $b_\epsilon(x,0)$ are smooth, the standard theory on the 2D viscous MHD equations (see e.g. \cite{SeTe}) guarantees that (\ref{ueq_ep})-(\ref{in_ep}) has a unique global smooth solution $(u_\epsilon, b_\epsilon)$. It is easy to see that $(u_\epsilon, b_\epsilon)$ obeys the {\it a priori} bounds in Propositions \ref{prop21} and \ref{prop22} uniformly in $\epsilon$. The solution $(u, \;b)$ of (\ref{ueq})-(\ref{divb}) is then obtained as a limit of $(u_\epsilon, \;b_\epsilon)$ and obey the bounds in Propositions \ref{prop21} and \ref{prop22}.

\vskip.1in
The uniqueness of the solutions follows from the elementary inequalities (see Lemma 14 of \cite{DP})
$$
\|f\|_\infty \le C\;(\|f\|_2 + \|f_x\|_2 + \|f_{yy}\|_{2}) \quad\mbox{and}\quad
\|f\|_\infty \le C\;(\|f\|_2 + \|f_y\|_2 + \|f_{xx}\|_2).
$$
In fact, applying these inequalities, we have
\begin{eqnarray*}
&& \int_0^t \left(\|\om(\tau)\|_\infty \;+\; \|j(\tau)\|_\infty\right) \;d\tau \\
&& \qquad \le \int_0^t \left(\|\om(\tau)\|_2 \;+\;\|\om_y(\tau)\|_2\; + \;\|\na\om_y(\tau)\|_2 \right) \; d\tau \\
&& \qquad \quad + \; \int_0^t \left(\|j(\tau)\|_2 \;+\; \|j_x(\tau)\|_2 +\; \|\na j_x(\tau)\|_2 \right)\; d\tau <\infty
\end{eqnarray*}
for any $t>0$. It is well-known (see e.g. \cite{CKS}, \cite{Wu}) that this bound yields the uniqueness.
\end{proof}

\subsection{(\ref{ueq})-(\ref{divb}) with $\nu_1=\nu>0$, $\nu_2=0$, $\eta_1=0$ and $\eta_2=\eta>0$} A global regularity result similar to Theorem \ref{major1} can be established for the 2D MHD equations (\ref{ueq})-(\ref{divb}) with $\nu_1=\nu>0$, $\nu_2=0$, $\eta_1=0$ and $\eta_2=\eta>0$.
\begin{theorem} \label{major3}
Consider the 2D MHD equations (\ref{ueq})-(\ref{divb}) with $\nu_1=\nu>0$, $\nu_2=0$, $\eta_1=0$ and $\eta_2=\eta>0$. Assume $u_0\in H^2({\mathbf R}^2)$ and $b_0\in H^2({\mathbf R}^2)$ with $\nabla \cdot u_0=0$ and $\nabla \cdot b_0=0$. Then (\ref{ueq})-(\ref{divb}) has a unique global classical solution $(u,b)$.  In addition, $(u, b)$ satisfies
\begin{equation}
(u, b) \in L^\infty([0,\infty); H^2), \qquad \om_x \in L^2([0,\infty); H^1), \qquad j_y \in L^2([0,\infty); H^1),
\end{equation}
where $\omega=\na \times u$ and $j=\na \times b$ represent the vorticity and the current density, respectively.
\end{theorem}
\begin{proof}
Although this theorem can be proven in a similar fashion as that of Theorem \ref{major1}, we provide an alternative proof. The idea is to convert (\ref{ueq})-(\ref{divb}) with $\nu_1=\nu>0$, $\nu_2=0$, $\eta_1=0$ and $\eta_2=\eta>0$ into a form dealt with by Theorem \ref{major1}. Assume that
$(u,b)$ solves (\ref{ueq})-(\ref{divb}) with $\nu_1=\nu>0$, $\nu_2=0$, $\eta_1=0$ and $\eta_2=\eta>0$. Set
\begin{eqnarray*}
&& U_1(x,y,t) = u_2(y,x,t), \quad U_2(x,y,t)=u_1(y,x,t),\quad P(x,y,t) = p(y,x,t), \\ && B_1(x,y,t) =b_2(y,x,t), \quad B_2(x,y,t) =b_1(y,x,t).
\end{eqnarray*}
Then $U=(U_1,U_2)$, $P$ and $B=(B_1,B_2)$ satisfy
\begin{eqnarray}
&& U_t + U\cdot\nabla U =-\nabla P + \nu\, U_{yy} + B\cdot\nabla B,   \label{ueq3}  \\
&& B_t + U\cdot\nabla B =\eta\; B_{xx}  + B\cdot\nabla U,  \label{beq3}   \\
&& \nabla\cdot U =0   \label{divu3}, \\
&& \nabla \cdot B=0. \, \label{divb3}
\end{eqnarray}
The global regularity of (\ref{ueq3})-(\ref{divb3}) guaranteed by Theorem \ref{major1} allows us to obtain the global regularity for (\ref{ueq})-(\ref{divb}) with $\nu_1=\nu>0$, $\nu_2=0$, $\eta_1=0$ and $\eta_2=\eta>0$. This completes the proof of Theorem \ref{major3}.
\end{proof}

\vskip.3in
\section{The MHD with magnetic diffusion}

This section focuses on (\ref{ueq})-(\ref{divb}) with $\nu_1=\nu_2=0$ and $\eta_1=\eta_2=\eta>0$. Two major results are established. The first is the global existence of a weak solution and the second assesses the global regularity and uniqueness of the weak solution under a suitable condition.
\begin{theorem} \label{major2}
Consider (\ref{ueq})-(\ref{divb}) with $\nu_1=\nu_2=0$ and $\eta_1=\eta_2=\eta>0$. Assume that $(u_0, b_0)\in H^1$ with $\na\cdot u_0=0$ and $\na \cdot b_0=0$. Then (\ref{ueq})-(\ref{divb}) has a global weak solution $(u,b)$ satisfying
\begin{equation}\label{regclass}
u\in C([0,\infty); H^1), \quad b\in C([0,\infty); H^1) \cap L^2([0,\infty); H^2).
\end{equation}
\end{theorem}
The proof of this result relies on a global {\it a priori} bound for $\om=\na\times u$ and $j=\na \times b$.

\begin{theorem} \label{cond_reg}
Assume the initial data $(u_0, b_0)\in H^3$, $\nabla\cdot u_0=0$ and $\na \cdot b_0=0$. Let $(u, b)$ be the corresponding solution of (\ref{ueq})-(\ref{divb}) with $\nu_1=\nu_2=0$ and $\eta_1=\eta_2=\eta>0$.  If, for some $T>0$,
\begin{equation}\label{supp}
\sup_{p\ge 2} \; \frac{1}{\sqrt{p}} \,\int_0^T \|\na u(t)\|_p \; dt < \infty,
\end{equation}
then $(u, b)$ is regular on $[0,T]$, namely
$$
(u, b) \in C([0,T]; H^3).
$$
In addition, two weak solutions $(u,b)$ and $(\tilde{u}, \tilde{b})$ in the regularity class (\ref{regclass}) must be identical on the time interval $[0,T]$ if $u$ satisfies (\ref{supp}).
\end{theorem}

The rest of this section is divided into four subsections. The first subsection presents a global {\it a priori} bound for $\|u\|_{H^1}$ and $\|b\|_{H^1}$ and the second proves Theorem \ref{major2}. The third subsection establishes a logarithmic Sobolev inequality, which serves as a preparation for the proof of Theorem \ref{cond_reg}. The last subsection proves Theorem \ref{cond_reg}.

\vskip.1in
\subsection{An {\it a priori} bound for $\|\nabla u\|_2$ and $\|\nabla b\|_2$}

\begin{proposition} \label{prop3}
If $(u,b)$ solves the 2D MHD equations (\ref{ueq})-(\ref{divb}) with $\nu_1=\nu_2=0$ and $\eta_1=\eta_2=\eta>0$, then, for any $t>0$,
\begin{equation} \label{ap}
\|\om(t)\|_2^2 + \|j(t)\|_2^2 + \eta \int_0^t \|\na j\|_2^2 \;d\tau \le C(\eta)\; (\|\na u_0\|_2^2 + \|\na b_0\|_2^2),
\end{equation}
where $C(\eta)$ is a constant depending on $\eta$ only. Therefore,
\begin{equation}\label{h1ap}
\|u(t)\|_{H^1}^2 + \|b(t)\|_{H^1}^2 + \eta \int_0^t \|b\|_{H^2}^2 \;d\tau \le C(\eta)\; (\|u_0\|_{H^1}^2 + \|b_0\|_{H^1}^2).
\end{equation}
\end{proposition}
\begin{proof} It follows easily from (\ref{ueq}) and (\ref{beq}) that, for any $t>0$,
\begin{equation}\label{l2}
\|u(t)\|^2_2 + \|b(t)\|_2^2 + 2 \eta \int_0^t \|\nabla b (\tau)\|_2^2 \; d\tau =\|u(0)\|^2_2 + \|b(0)\|_2^2.
\end{equation}
To prove (\ref{ap}), we employ the equations of the vorticity $\om$ and the current density $j$,
\begin{eqnarray}
&&
\om_t + u\cdot\nabla \om = b\cdot\nabla j,  \label{diffv}\\
&&
j_t + u\cdot\nabla j = \eta\; \Delta j + b\cdot\nabla \om  + 2\pp_xb_1(\pp_x u_2 + \pp_y u_1)-2 \pp_x u_1 (\pp_x b_2 + \pp_y b_1).  \label{diffj}
\end{eqnarray}
Taking the inner products of (\ref{diffv}) with $\om$ and of (\ref{diffj}) with $j$, we find
\begin{eqnarray*}
&&
\frac12\,\frac{d\, \|\om\|_2^2}{dt} \,=\,\int b\cdot\nabla j \,\om\,dxdy,\\
&&
\frac12\,\frac{d\, \|j\|_2^2}{dt} + \eta \|\nabla j\|^2_2  = \int b\cdot\nabla \om \,j \,dxdy +\; 2\;\int (\pp_xb_1(\pp_x u_2 + \pp_y u_1)- \pp_x u_1 (\pp_x b_2 + \pp_y b_1))\,j\,dxdy.
\end{eqnarray*}
Since
$$
\int b\cdot\nabla j \;\om\,dxdy + \int b\cdot\nabla \om \;j \,dxdy =0,
$$
we have, for $X(t)= \|\om(t)\|_2^2 + \|j(t)\|_2^2$,
\begin{equation*}
\frac{d\,X(t)}{dt}  + 2\eta \, \|\nabla j\|^2_2 \le 8 \|\na u\|_2\,\|\na b\|_4 \,\|j\|_4,
\end{equation*}
where we have applied the H\"{o}lder inequality. Applying the inequalities
\begin{equation*}
\|\nabla u\|_2 \le \|\om\|_2, \quad \|\nabla b\|_4 \le \|j\|_4, \quad \|j\|^2_4 \le \|j\|_2\, \|\nabla j\|_2
\end{equation*}
and Young's inequality, we find
\begin{equation*}
\frac{d\,X(t)}{dt}  + 2\eta \, \|\nabla j\|^2_2 \le \frac{16}{\eta} \|\omega\|^2_2\;\|j\|_2^2 \, +\eta \;\|\nabla j\|^2_2.
\end{equation*}
In particular,
\begin{eqnarray*}
&&\hskip-.8in
\frac{d\,X(t)}{dt}  + \eta \, \|\nabla j\|^2_2 \le \frac{16}{\eta}\|j\|_2^2\,\,  X(t).
\end{eqnarray*}
By Gronwall's inequality,
\begin{eqnarray*}
X(t) + \eta\,\int_0^t  \|\nabla j(\tau)\|^2_2\,d\tau \le X(0)\, \exp\left(\frac{16}{\eta}\int_0^t \|j\|_2^2 \; d\tau \right),
\end{eqnarray*}
which, together with (\ref{l2}), yields (\ref{ap}) and (\ref{h1ap}).
\end{proof}

\vskip.1in
\subsection{Proof of Theorem \ref{major2}} Let $\epsilon>0$ be a small parameter and consider the regularized system of equations
\begin{eqnarray*}
&& \pp_t\; u_\epsilon + u_\epsilon\cdot\nabla u_\epsilon =-\nabla p_\epsilon + \epsilon \; \Dd u_\epsilon + b_\epsilon\cdot\nabla b_\epsilon,  \\
&& \pp_t\; b_\epsilon + u_\epsilon\cdot\nabla b_\epsilon =\eta\; \Dd b_\epsilon + b_\epsilon \cdot\nabla u_\epsilon,  \\
&& \nabla\cdot u_\epsilon =0,    \\
&& \nabla \cdot b_\epsilon =0.
\end{eqnarray*}
This system of equations admits a unique global solution $(u_\epsilon, b_\epsilon)$ that satisfies the global {\it a priori} bound stated in Proposition \ref{prop3} uniformly in terms of $\epsilon$. By going through a standard limit process, we conclude that $(u_\epsilon, b_\epsilon)$ converge to a weak solution of (\ref{ueq})-(\ref{divb}) with $\nu_1=\nu_2=0$ and $\eta_1=\eta_2=\eta$. This completes the proof of Theorem \ref{major2}.

\vskip.1in
\subsection{A logarithmic inequality} This subsection presents a logarithmic Sobolev inequality, which plays an important role in the proof of Theorem  \ref{cond_reg}. A similar inequality was previously obtained by Danchin and Paicu \cite{DP} and their proof involves tools from Fourier analysis such as the Littlewood-Paley decomposition. The proof presented here is different and more elementary.

\begin{lemma} \label{logsobolev}
For any function of two variables $f=f(x)$, $x\in {\mathbf R}^2$, the following logarithmic inequality holds
$$
\|f\|_{L^\infty} \le C\, \sup_{q\ge 2}\frac{\|f\|_{q}}{\sqrt{q}}\, \left[\ln(e + \|f\|_{H^2})\right]^\frac12.
$$
\end{lemma}

\begin{proof} We follow the approach of Hou and Li \cite{HL}. Denote by $B_r$ the disk centered at the origin with radius $r$. Let $\phi\in C^\infty({\mathbf R}^2)$ be a smooth cutoff function satisfying
$$
\phi(0)=1, \quad |\nabla \phi| \le C, \quad |\Delta \phi| \le C, \quad \mbox{supp} \;\phi \subset B_1.
$$
Set $w=f \phi$. According to the solution formula of the 2D Laplace equation, we have, for any $p\ge 2$,
\begin{eqnarray*}
w^p(0) &=&\frac{1}{2\pi} \int_{B_\epsilon} (\ln |y| -\ln \epsilon) \Delta w^p (y)\, dy + \frac{1}{2\pi} \int_{B_1\setminus B_\epsilon} (\ln |y| -\ln \epsilon) \Delta w^p (y)\, dy \\
&=& I+ II.
\end{eqnarray*}
Since
$$
\Delta w^p = p \;w^{p-1} \;\Delta w + p(p-1)\; w^{p-2}\; |\nabla w|^2,
$$
we obtain by applying H\"{o}lder's inequality
$$
|I| \le \frac{p}{2\pi} \, \epsilon^{\frac23}\, \left[\|\Delta w\|_{2}\,
\|w\|_{{6(p-1)}}^{p-1} + (p-1)\,\|\nabla w\|_{4}^2 \,\|w\|_{{6(p-2)}}^{p-2}\right].
$$
By the embedding inequality
$$
\|\nabla w\|_{4} \le C \|w\|_{2}^\frac14 \, \|\Delta w\|_{2}^\frac34,
$$
we have, for $C$ independent of $p$,
$$
|I| \le C p\, \epsilon^{\frac23}\, \|\Delta w\|_{2}\, \|w\|_{{6(p-1)}}^{p-1}
+ C p(p-1)\,\epsilon^{\frac23}\, \|w\|_{2}^\frac12 \, \|\Delta w\|_{2}^\frac32\, \|w\|_{{6(p-2)}}^{p-2}.
$$
Integrating by parts in $II$  yields
$$
II = \frac{p}{2\pi} \int_{B_1\setminus B_\epsilon} w^{p-1} \, \frac{y\cdot \nabla w}{|y|^2} \,dy.
$$
By H\"{o}lder's inequality,
$$
|II| \le C\, p \, \left(\ln\frac1{\epsilon} \right)^\frac12\, \|\nabla w\|_{4}\, \|w\|_{{4(p-1)}}^{p-1} \le  C  p \, \left(\ln\frac1{\epsilon} \right)^\frac12\, \|w\|_{2}^\frac14 \,\|\Delta w\|_{2}^\frac34\, \|w\|_{{4(p-1)}}^{p-1}.
$$
Now, set
$$
\epsilon^\frac23 \|\Delta w\|_{2} =1 \quad \mbox{or}\quad \epsilon = \|\Delta w\|_{2}^{-\frac32} \quad\mbox{and}\quad  p = \ln\frac1{\epsilon}.
$$
\vspace{.2in}
We then have
\begin{eqnarray*}
|w(0)| &\le& C\, p^\frac{1}{p} \,\|w\|_{{6(p-1)}}^{1-\frac1p} + C (p(p-1))^{\frac1p} \, \|w\|_{2}^\frac1{2p} \,\|\Delta w\|_{2}^\frac1{2p}\,
\|w\|_{{6(p-2)}}^{1-\frac2p}
\\
&& + \,C\, p^{\frac3{2p}}\,\|w\|_{2}^\frac1{4p} \,\|\Delta w\|_{2}^{\frac3{4p}}\,\|w\|_{{4(p-1)}}^{1-\frac1p}.
\end{eqnarray*}
Use the fact that $p^\frac{1}{p}<C$, $(p(p-1))^{\frac1p}<C$, and
$$
\|\Dd w\|_{2}^\frac1{2p} = \epsilon^{\frac{1}{3\ln \epsilon}} = e^\frac13, \quad \|w\|_{q} \le \sqrt{q}\, \sup_{q\ge 2} \frac{\|w\|_{q}}{\sqrt{q}},
$$
we obtain that
$$
|w(0)| \le C \,\sup_{q\ge 2} \frac{\|w\|_{q}}{\sqrt{q}} \ln^\frac12(e + \|\Delta w\|_{2}).
$$
Noticing that
$$
|f(0)=|w(0)| \quad \mbox{and}\quad \|\Delta w\|_{2} \le C(\|f\|_{2} + \|\Delta f\|_{2}) \le C \|f\|_{H^2},
$$
we conclude the proof of Lemma \ref{logsobolev}.
\end{proof}

\vskip.1in
\subsection{Proof of Theorem \ref{cond_reg}}  To show the regularity, we bound $\|(u, b)\|_{H^3}$. According to Proposition \ref{prop3}, $\|(u, b)\|_{H^1}$ admits a global uniform bound. Now, consider $\nabla \om$ and $\nabla j$, which satisfy
\begin{eqnarray}
\pp_t \nabla\om + u\cdot \na (\na\om) &=& -(\nabla u ) \, \na \om  + b\cdot \na (\na j) + (\na b )\, \na j,  \nonumber \\
\pp_t \nabla j  + u\cdot \na (\na j) &=& \eta \, \Delta (\nabla j) -(\nabla u ) \, \na j  + b\cdot \na (\na \om) + (\nabla b)\,\nabla \om \nonumber \\
&& + \;2\,\nabla[\pp_x b_1 (\pp_x u_2 + \pp_y u_1)] - 2 \nabla [\pp_x u_1 (\pp_x b_2 + \pp_y b_1)]. \nonumber
\end{eqnarray}
Therefore,
\begin{eqnarray*}
&&
\frac12\,\frac{d}{dt} \left(\|\na \om\|_2^2 + \|\na j\|_2^2 \right) + \eta \|\Dd j\|_2^2
=-\int \na \om\cdot\na u\cdot \na \om -\int \na u\cdot \na j\cdot \na j \, +\, 2 \int \na b\cdot \na j \cdot \na \om  \\
&& \qquad \qquad \qquad\quad
+ \, 2 \int \nabla [\pp_x b_1(\pp_y u_1 + \pp_x u_2)]\cdot \na j -2 \int \nabla [\pp_x u_1 (\pp_y b_1 + \pp_x b_2)]\cdot \nabla j \\
&& \qquad \qquad \qquad
 \equiv K_1 +K_2 +K_3 +K_4+K_5.
\end{eqnarray*}
The terms on the right can be estimated as follows.
\begin{eqnarray*}
K_1 &\le & \|\na u\|_\infty\; \|\na \om\|_2^2. \\
K_2 &=& -\int \na u\cdot \na j\cdot \na j  \le \|\na u\|_2 \|\na j\|_4^2
\\
&\le & C \, \|\na u\|_2 \|\na j\|_2 \|\Dd j\|_2\\
&\le & \frac{\eta}{8} \|\Dd j\|_2^2 \,+\, C\; \|\om\|_2^2 \|\na j\|_2^2.
\end{eqnarray*}

\begin{eqnarray*}
K_3 &=& 2\int \na b\cdot \na \om\cdot \na j  \le 2\|\na \om\|_2 \;\|\na b\|_4 \|\na j\|_4 \\
&\le& C \,\|\na \om\|_2 \; \|\na b\|_2^\frac12\; \|\Dd b\|_2^\frac12\; \|\na j\|_2^\frac12\; \|\Dd j\|_2^\frac12 \\
&\le& \frac{\eta}{8} \|\Dd j\|_2^2 \,+\, C\;\|\na \om\|_2^\frac43\; \|\na b\|_2^\frac23\; \|\na j\|_2^\frac43 \\
&\le & \frac{\eta}{8} \|\Dd j\|_2^2 \,+\, C\;\|j\|_2^\frac23\; \|\nabla j\|_2^\frac23\; (\|\na \om\|_2^2 + \|\na j\|_2^2).
\end{eqnarray*}

\begin{eqnarray*}
K_4 &=& 2 \int \nabla [\pp_x b_1(\pp_y u_1 + \pp_x u_2)]\cdot \na j \\
&=& 2 \int \pp_x \nabla b_1 \cdot \na j \;(\pp_y u_1 + \pp_x u_2) + 2 \int \pp_x b_1 (\pp_y \na\,u_1 + \pp_x \na\,u_2)\cdot \na j \\
&\le& 4\int |\na j|^2 |\na u| \;+\; 4 \int |\na b|\;|\na \om|\; |\na j|\\
&\le & \frac{\eta}{4} \|\Dd j\|_2^2 \,+\, C\;\|\om\|_2^2 \|\na j\|_2^2 + C\;\|j\|_2^\frac23\; \|\nabla j\|_2^\frac23\; (\|\na \om\|_2^2 + \|\na j\|_2^2).
\end{eqnarray*}
Putting together these estimates, we have
\begin{eqnarray*}
&&
\frac{d}{dt} \left(\|\na \om\|_2^2 + \|\na j\|_2^2 \right) + \eta \|\Dd j\|_2^2 \\
&&
\qquad \quad \le \|\nabla u\|_\infty \, \|\na \om\|_2^2  +  C\;\|\om\|_2^2 \|\na j\|_2^2 + C\;\|j\|_2^\frac23\; \|\nabla j\|_2^\frac23\; (\|\na \om\|_2^2 + \|\na j\|_2^2).
\end{eqnarray*}
We now bound the third-order derivatives of $(u, b)$. For any multi-index $\beta$  with $|\beta|=3$, $D^\beta u$ and $D^\beta b$ satisfy
\begin{eqnarray*}
&& \pp_t\, D^\beta u + u\cdot \na D^\beta u = -\na D^\beta p + b\cdot \na D^\beta b
   - [D^\beta, u\cdot\na] u + [D^\beta, b\cdot\na] b, \\
&& \pp_t\, D^\beta b + u\cdot \na D^\beta b = \eta\, \Delta D^\beta b + b\cdot \na D^\beta u - [D^\beta, u\cdot\na] b + [D^\beta, b\cdot\na],
\end{eqnarray*}
where $[D^\beta, f\cdot\na] g=D^\beta(f\cdot\na g) - f \cdot\na D^\beta g$. Taking the inner products of these equations with $D^\beta u$ and $D^\beta b$, respectively, and integrating by parts, we have
$$
\frac12 \frac{d}{dt} \left(\|D^\beta u\|_2^2 + \|D^\beta b\|_2^2\right) + \eta \|\na D^\beta b\|_2^2  =L_1 + L_2 + L_3 +L_4
$$
where
$$
L_1 = -([D^\beta, u\cdot\na] u, D^\beta u), \qquad L_2 = ([D^\beta, b\cdot\na] b, D^\beta u),
$$
$$
L_3 = -([D^\beta, u\cdot\na] b, D^\beta b), \qquad L_4 = ([D^\beta, b\cdot\na]b, D^\beta b).
$$
To bound $L_1$, $L_2$, $L_3$ and $L_4$, we recall the commutator estimate (see \cite[p.334]{KPV})
\begin{equation} \label{comm}
\| [D^\beta, f\cdot\na] g \|_p \le C\, (\|\na f\|_{p_1} \,\|\na g\|_{W^{2,p_2}} + \|f\|_{W^{3,p_3}} \, \|\na g\|_{p_4})
\end{equation}
valid for any $p$, $p_2$, $p_3\in (1,\infty)$ and $\frac{1}{p} = \frac{1}{p_1}+\frac{1}{p_2} = \frac{1}{p_3} + \frac{1}{p_4}$. Applying this inequality, we obtain
\begin{eqnarray*}
|L_1 | &\le& \|[D^\beta, u\cdot\na] u\|_2 \, \|D^\beta u\|_2 \le C\; \|\na u\|_\infty \|u\|_{H^3}\, \|D^\beta u\|_2, \\
|L_2 | &\le& \|[D^\beta, b\cdot\na] b\|_2\, \|D^\beta u\|_2 \le C\; (\|\na b \|_4\,\|\na b \|_{W^{2,4}} + \|b\|_{W^{3,4}}\, \|\na b\|_4)\, \|D^\beta u\|_2.
\end{eqnarray*}
By the basic calculus inequality, for any $f\in H^1({\mathbf R}^2)$,
\begin{equation} \label{fb}
\|f\|_4 \le C\, \|f\|_2^\frac12\, \|\nabla f\|_2^\frac12,
\end{equation}
we have
$$
|L_2| \le C \; \|\na b\|_2^\frac12\, \|\Delta b\|_2^\frac12\, \|b\|_{H^3}^\frac12 \, \|\nabla b\|_{H^3}^\frac12 \, \|D^\beta u\|_2.
$$
By Young's inequality,
\begin{eqnarray*}
|L_2| &\le& \frac{\eta}{4}\, \|\nabla b\|_{H^3}^2 + C\, \|\na b\|_2^\frac23\, \|\Delta b\|_2^\frac23\, \|b\|_{H^3}^\frac23 \, \|D^\beta u\|_2^\frac43 \\
&\le& \frac{\eta}{4}\, \|\nabla b\|_{H^3}^2 + C\, \|\na b\|_2^\frac23\, \|\Delta b\|_2^\frac23\, (\|b\|_{H^3}^2 + \|D^\beta u\|_2^2).
\end{eqnarray*}
By (\ref{comm}) again,
$$
|L_3| \le \|[D^\beta, u\cdot\na] b\|_{\frac43}\, \|D^\beta b\|_4 \le C\, \left(\|\na u\|_2 \,\|\na b\|_{W^{2,4}} \,+\,\|u\|_{H^3}\, \|\na b\|_4 \right)\, \|D^\beta b\|_4.
$$
Therefore,
\begin{eqnarray*}
|L_3| & \le & C\, \|\om\|_2\, \|b\|_{H^3}\, \|\nabla b\|_{H^3} + C \; \|\na b\|_2^\frac12\, \|\Delta b\|_2^\frac12\, \|b\|_{H^3}^\frac12 \, \|\nabla b\|_{H^3}^\frac12 \,\|u\|_{H^3} \\
& \le &  \frac{\eta}{4}\, \|\nabla b\|_{H^3}^2 + C\,\|\om\|^2_2\, \|b\|^2_{H^3} +C\; \|\na b\|_2^\frac23\, \|\Delta b\|_2^\frac23\, (\|b\|_{H^3}^2 + \| u\|_{H^3}^2).
\end{eqnarray*}
Similarly, $L_4$ is bounded as follows.
$$
|L_4| \le \frac{\eta}{4}\, \|\nabla b\|_{H^3}^2 + C\, \|\na b\|_2^\frac23\, \|\Delta b\|_2^\frac23\, \|b\|_{H^3}^2.
$$
Combining all these estimates, we obtain
$$
\frac{d}{dt} (\|u\|^2_{H^3} + \|b\|^2_{H^3}) + \eta\, \|\na b\|^2_{H^3} \le C \|\nabla u\|_\infty\, \|u\|^2_{H^3} + \,\|\om\|^2_2\, \|b\|^2_{H^3} +C\;\|j\|_2^\frac23 \|\nabla j\|_2^\frac23\,\|b\|^2_{H^3}.
$$
Applying Lemma \ref{logsobolev} to bound $\|\na u\|_\infty$, we obtain the regularity part of Theorem \ref{cond_reg}.

\vspace{.1in}
To prove the uniqueness, we consider the difference
$$
(W, B) = (\tilde{u}, \tilde{b})- (u,b),
$$
which satisfies the equations
\begin{eqnarray}
&&
W_t + \tilde{u}\cdot\nabla W + W\cdot\nabla u = -\nabla P + \tilde{b}\cdot\nabla B + B\cdot\nabla b \label{diff1},  \\
&&
B_t + \tilde{u}\cdot\nabla B + W\cdot\nabla b =\eta \Dd B + \tilde{b}\cdot\nabla W + B\cdot\na u, \label{diff2}
\end{eqnarray}
where $P$ is the difference between the corresponding pressures. Adding the inner products of (\ref{diff1}) with $W$ and of (\ref{diff2}) with $B$ and integrating by parts, we obtain
\begin{eqnarray}
\frac12\,\frac{d}{dt} \left(\|W\|_2^2 + \|B\|_2^2 \right) + \eta \|\na B\|_2^2 &\le & \int |W\cdot\nabla u \cdot W| + \int |B\cdot\na u\cdot B| +2 \int |W|\,|\na b|\,|B| \nonumber\\
&\le& \|\na u\|_\infty \; \left(\|W\|_2^2 + \|B\|_2^2 \right) + 2 \, \|W\|_2\,\|B\|_4\,\|\na b\|_4. \label{wd}
\end{eqnarray}
By (\ref{fb}), we have
\begin{eqnarray*}
2\, \|W\|_2\,\|B\|_4\,\|\na b\|_4 &\le& C\,\|W\|_2 \|B\|_2^\frac12 \, \|\na B\|_2^\frac12 \,\|\na b\|_2^\frac12\, \|\Delta b\|_2^\frac12 \\
&\le& \frac{\eta}{2}\,\|\na B\|_2^2 + C\, \|W\|_2^\frac43\, \|B\|_2^\frac23\, \|\na b\|_2^\frac23 \|\Delta b\|_2^\frac23 \\
&\le& \frac{\eta}{2}\,\|\na B\|_2^2 + C\, \|\na b\|_2^\frac23 \|\Delta b\|_2^\frac23 \left(\|W\|_2^2 + \|B\|_2^2 \right).
\end{eqnarray*}
Inserting the above estimate in (\ref{wd}) and applying Lemma \ref{logsobolev} to bound $\|\na u\|_\infty$, we obtain the desired uniqueness. This completes the proof of Theorem \ref{major2}.

 \vspace{.4in}
\section*{Acknowledgments}
Cao is partially supported by NSF grant DMS 0709228 and a FIU foundation. Wu is partially supported by the AT \& T Foundation at OSU.


\begin{thebibliography}{99}
\bibitem{CKS} R. Caflisch,I. Klapper and G. Steele, Remarks on singularities, dimension and energy dissipation for ideal hydrodynamics and MHD, {\it Comm. Math. Phys.} {\bf  184} (1997),  443--455.

\bibitem{Ch} D. Chae, Global regularity for the 2D Boussinesq equations with partial viscosity terms, {\em Advances in Math.} {\bf 203} (2006), 497-513.

\bibitem{DP} R. Danchin and M. Paicu, Global existence results for the anisotropic Boussinesq system in dimension two, arXiv: 0809.4984v1 [math.AP] 19 Sep 2008.

\bibitem{DL} G. Duvaut and  J.-L. Lions, In\'{e}quations en thermo\'{e}lasticit\'{e} et magn\'{e}tohydrodynamique, {\it Arch. Rational
Mech. Anal.} {\bf  46} (1972), 241-279.

\bibitem{HL} {\scshape T. Hou and C. Li}, Global well-posedness of the viscous Boussinesq equations, {\em Discete and Continuous Dynamics Systems} {\bf 12} (2005), 1-12.

\bibitem{KPV} C.E. Kenig, G. Ponce and L. Vega, Well-posedness of the initial value problem for the Korteweg-de Vries equation, {\it J. Amer. Math. Soc. \bf  4 }(1991),  323--347.

\bibitem{SeTe} M. Sermange and R. Temam, Some mathematical questions related to the MHD equations, {\it Comm. Pure Appl. Math.}  {\bf 36}  (1983),  635--664.

\bibitem{Wu} J. Wu, Viscous and inviscid magnetohydrodynamics equations, {\it J. d'Analyse Math.}\,\,{\bf  73}  (1997), 251--265.
\end{thebibliography}
\end{document}